\documentclass[leqno,11pt]{amsart}
\usepackage{amsmath,amstext,amssymb,amsopn,amsthm,mathrsfs,wasysym}

\theoremstyle{plain}

\allowdisplaybreaks

\newtheorem{thm}{Theorem}[section]
\newtheorem{cor}[thm]{Corollary}
\newtheorem{pro}[thm]{Proposition}

\newtheorem{lem}[thm]{Lemma}

\newtheorem*{ex*}{Example}

\newcommand{\N}{\mathbb{N}}

\newcommand{\R}{\mathbb{R}}

\def\P{\mathcal P}
\def\m{\mu} 						
\def\J{\mathcal J} 			

\def\a{\alpha}
\def\b{\beta}
\def\ab{\alpha,\beta}
\def\ba{\beta,\alpha}
\def\t{\theta}
\def\v{\varphi}
\def\vp{\varphi}
\def\st{\sin \frac{\theta}{2}}
\def\svp{\sin \frac{\varphi}{2}}
\def\ste{\sin \frac{\theta + \eta}{2}}
\def\stpe{\sin \frac{\theta' + \eta}{2}}
\def\ct{\cos \frac{\theta}{2}}
\def\cvp{\cos \frac{\varphi}{2}}
\def\cte{\cos \frac{\theta + \eta}{2}}
\def\ctpe{\cos \frac{\theta' + \eta}{2}}
\def\pia{d\Pi_{\alpha}(u)}
\def\pib{d\Pi_{\beta}(v)}

\def\ot{\Omega_{\alpha,\beta}(\theta,\eta,t)}
\def\ott{\Omega_{\alpha,\beta}(\theta',\eta,t)}
\def\ottt{\Omega_{\alpha,\beta}}

\DeclareMathOperator{\support}{supp}
\DeclareMathOperator{\id}{Id}

\def\piK{d\Pi_{\alpha, K}(u)}
\def\piR{d\Pi_{\beta, R}(v)}
\def\abpiK{\Pi_{\alpha, K}}

\def\stt{\sin \frac{\widetilde{\theta}}{2}}

\def\ctt{\cos \frac{\widetilde{\theta}}{2}}

\def\lam{\lambda}

\title[Lusin area integrals related to Jacobi expansions]{Lusin area integrals related to Jacobi expansions}

\author[T.Z. Szarek]{Tomasz Z. Szarek}
\address{Tomasz Z. Szarek,     \newline
			Institute of Mathematics,
		Polish Academy of Sciences, \newline
      \'Sniadeckich 8,
      00--656 Warszawa, Poland
      }
\email{szarektomaszz@gmail.com}

\begin{document}

\begin{abstract}
We investigate mixed Lusin area integrals associated with Jacobi trigonometric polynomial expansions. We prove that these operators can be viewed as vector-valued Calder\'on-Zygmund operators in the sense of the associated space of homogeneous type.
Consequently, their various mapping properties, in particular on weighted $L^p$ spaces, follow from the general theory.
\end{abstract}

\maketitle

\footnotetext{
\emph{\noindent 2010 Mathematics Subject Classification:} primary 42C05; secondary 42C10\\
\emph{Key words and phrases:} Jacobi polynomial, Jacobi expansion, Jacobi operator,
			Jacobi-Poisson kernel,
			Jacobi-Poisson semigroup, square function, Lusin area integral, 
			Calder\'on-Zygmund operator. \\
			\indent	
Research supported by the 
National Science Centre of Poland, project no.\ 2012/05/N/ST1/02746.
}

\section{Introduction} \label{sec:intro}

One of the principal aims of the papers \cite{NoSj} and \cite{NoSjSz} was to prove that several fundamental harmonic analysis operators in Jacobi trigonometric polynomial expansions are (vector-valued) Calder\'on-Zygmund operators. That research included such operators as higher order Riesz transforms, multipliers of Laplace and Laplace-Stieltjes transform types, Jacobi-Poisson semigroup maximal operator and Littlewood-Paley-Stein type mixed $g$-functions. 
This article is a continuation and completion of the research performed in \cite{NoSj,NoSjSz}. Motivated by the comment in \cite[p.\,187]{NoSjSz} in the present paper, we study mixed Lusin area integrals from a similar perspective. These objects have more complex structure than those mentioned above, therefore their treatment is considerably more involved and demands more effort and additional technical tools. We point out that analysis in various Jacobi settings received a considerable attention in recent years, see for instance \cite{CNS,L,NoSj,NoSj2,NoSjSz} as well as numerous other references given there.

In the last years Lusin area type integrals attracted attention of many mathematicians. These operators, sometimes called conical square functions, were also studied in some contexts of orthogonal expansions, see e.g.\ \cite{BCFR,BFRTT,BMR,CaSz,K,MNP,NoStSz,P,TZS}.
In the classical situation these objects turn out to be not only interesting on their own right, but also have significant applications. For instance, a variant of Lusin area integral was used by Segovia and Wheeden \cite{SW} to characterize potential spaces on $\R^d$, $d \ge 1$.
Inspired by that paper, the authors of \cite{BFRTT} showed, among other things, that a similar characterization is also possible in case of some Schr\"odinger operators' frameworks. 
We point out that quite recently some variant of Lusin area integral was investigated in the Ornstein-Uhlenbeck context in \cite{K,MNP,P} in connection with the Gaussian Hardy space theory point of view.
Thus our motivation to study mixed Lusin area integrals in Jacobi expansions comes also from their potential applications in further research.   

In this article we study two kinds of mixed Lusin area integrals $S_{M,N}^{\ab}$ and $\mathcal{S}_{M,N}^{\ab}$ (see Section \ref{sec:prel} for the definitions), which come from two different notions of higher order derivatives. The first one is simply a composition of the first order derivative and was for instance implemented in \cite{NoSj,NoSjSz}. The second one was used in \cite{L} and has its roots in the so-called symmetrization procedure proposed by Nowak and Stempak in \cite{NoSt}. We note that, in some aspects, the latter notion seems to be closer to the classical theory on the Euclidean spaces. To see this compare for example \cite[Proposition 2.5]{CNS} with \cite[Remark 3.8]{CNS}, where two kinds of higher order Riesz transforms are studied in the Jacobi context.

Our main result, see Theorem \ref{thm:CZ} below, says that the mixed Lusin area integrals of both kinds can be viewed as vector-valued Calder\'on-Zygmund operators in the sense of the associated space of homogeneous type. 
Consequently, their mapping properties follow from the general theory. The main difficulty connected with the Calder\'on-Zygmund theory approach is showing the related kernel estimates. Here the starting point is the method of proving standard estimates established in \cite{NoSj} for $\ab \ge -1/2$ and extended in \cite{NoSjSz} to all admissible $\ab > -1$. 
We point out that the crucial role in that method plays a convenient integral representation for the Jacobi-Poisson kernel obtained in \cite{NoSjSz}. 
Nevertheless, to treat Lusin area integrals, which have more complex nature than the operators considered in the above mentioned papers, some generalization of this technique is needed. The latter is of independent interest and is inspired by similar tools elaborated recently in other settings of orthogonal expansions, see \cite{CaSz,NoStSz,TZS}. 

The paper is organized as follows. In Section \ref{sec:prel} we introduce the context of Jacobi expansions and define the mixed Lusin area integrals. Further, we state the main result (Theorem \ref{thm:CZ}) and reduce its proof to showing the standard estimates for the related kernels (Theorem \ref{thm:kerest}). We conclude this section by giving comments connected with our main result. 
Section \ref{sec:kerest} contains preparatory facts and lemmas, which finally allows us to prove the relevant kernel estimates. This is the largest and the most technical part of the paper.

Throughout the paper we use a fairly standard notation with essentially all symbols referring to
the space of homogeneous type $((0,\pi),\mu_{\ab},|\cdot|)$, where
$|\cdot|$ is the Euclidean norm and $\mu_{\ab}$ is a measure on $(0,\pi)$ defined below. In particular, for $\t \in (0,\pi)$ and $r>0$ the ball $B(\t,r)$ is simply the interval $(\t - r, \t + r) \cap (0,\pi)$.
By $\langle f,g \rangle_{d\mu_{\ab}}$ we
mean $\int_{(0,\pi)} f\overline{g}\,d\mu_{\ab}$ whenever the integral makes sense.
Further, by $L^p(w d\mu_{\ab}) = L^p((0,\pi), w d\mu_{\ab})$ we understand the weighted
$L^p$ space with $w$ being a non-negative
weight on $(0,\pi)$. Furthermore, for $1 \le p < \infty$ we denote by $A_p^{\ab}$ the Muckenhoupt class
of $A_p$ weights connected with the space 
$((0,\pi),\mu_{\ab},|\cdot|)$, for the definition of $A_p^{\ab}$ see for instance \cite[p.\,720]{NoSj}.

When writing estimates, we will frequently use the notation $X \lesssim Y$ to indicate that
$X \le C Y$ with a positive constant $C$ independent of significant quantities. We shall write
$X \simeq Y$ when simultaneously $X \lesssim Y$ and $Y \lesssim X$.

\section{Preliminaries and statement of main result} \label{sec:prel}

As in \cite{NoSj,NoSjSz},
we consider expansions into Jacobi trigonometric polynomials.
For parameters $\ab > -1$ the normalized Jacobi trigonometric polynomials are given by
$$
\mathcal{P}_n^{\ab}(\t) = c_n^{\ab} P_n^{\ab}(\cos\t), \qquad \t \in (0,\pi),
\quad n\ge 0,
$$
where $c_n^{\ab}$ are suitable normalizing constants, and $P_n^{\ab}$ are the classical
Jacobi polynomials defined on the interval $(-1,1)$, see Szeg\H o's monograph \cite{Sz} or \cite{NoSj,NoSjSz}. 
It is well known that the system $\{\mathcal{P}_n^{\ab}: n \ge 0\}$ is an orthonormal basis in $L^2(d\mu_{\ab})$,
where $\mu_{\ab}$ is a measure on the interval $(0,\pi)$ defined by the density
$$
d\mu_{\ab}(\t) = \Big( \sin\frac{\t}2 \Big)^{2\alpha+1} \Big( \cos\frac{\t}2\Big)^{2\beta+1} d\t, \qquad \t \in (0,\pi).
$$
Further, $\mathcal{P}_n^{\ab}$ are eigenfunctions of the Jacobi differential operator
$$
\mathcal{J}^{\ab} = - \frac{d^2}{d\theta^2} - \frac{\alpha-\beta+(\alpha+\beta+1)\cos\theta}{\sin \theta}
    \frac{d}{d\theta} + \lam_0^{\ab}, \qquad \textrm{where} \qquad 
    \lam_0^{\ab} =  \left( \frac{\alpha+\beta+1}{2} \right)^2.
$$
More precisely,
$$
\mathcal{J}^{\ab} \mathcal{P}_n^{\ab} = \lam_n^{\ab} \mathcal{P}_n^{\ab},
\qquad \lam_n^{\ab} = \left( n + \frac{\alpha+\beta+1}{2} \right)^2,
    \qquad n \ge 0.
$$
We denote by the same symbol $\mathcal{J}^{\ab}$ the natural self-adjoint and non-negative extension in $L^2(d\mu_{\ab})$ given by
\[
\mathcal{J}^{\ab} f = \sum_{n=0}^{\infty} \lam_n^{\ab}
    \big\langle f,\P_n^{\ab}\big\rangle_{d\mu_{\ab}} \P_n^{\ab}
\]
on the domain consisting of all $f \in L^2(d\mu_{\ab})$ for which the above series converges in $L^2(d\mu_{\ab})$.

The Jacobi-Poisson semigroup generated by $-\sqrt{\J^{\ab}}$ is expressible via the spectral series
\[
\mathcal{H}_t^{\ab}f =
\sum_{n=0}^{\infty} e^{- t \sqrt{\lam_n^{\ab} } }
    \big\langle f,\P_n^{\ab}\big\rangle_{d\mu_{\ab}} \P_n^{\ab},
\qquad t \ge 0, \quad f \in L^2(d\mu_{\ab}).
\]
Further, it has the integral representation
\begin{align*}
\mathcal{H}_t^{\ab}f (\t) &=
\int_0^\pi H_t^{\ab}(\t,\v) f(\v) \, d\mu_{\ab}(\v), 
\qquad \t \in (0,\pi), \quad f \in L^2(d\mu_{\ab}),\\
H_t^{\ab}(\t,\v) &=
    \sum_{n=0}^{\infty} e^{-t \sqrt{\lam_n^{\ab} } } \P_n^{\ab}(\t) \P_n^{\ab}(\v),
		\qquad t>0, \quad \t,\v \in (0,\pi).
\end{align*}
It is worth noting that the series/integral defining $\mathcal{H}_t^{\ab}f (\t)$
converges pointwise and produces a smooth function of 
$(t,\t) \in (0,\infty) \times (0,\pi)$ for any 
$f \in L^p(wd\mu_{\ab})$, $w \in A_p^{\ab}$, $1 \le p < \infty$; for details see \cite[Section 2]{NoSj}.
Furthermore, the series defining the so-called Jacobi-Poisson kernel $H_t^{\ab}(\t,\v)$ is well defined in a pointwise sense for all $t > 0$, $\t,\v \in (0,\pi)$ and produces a smooth function of 
$(t,\t,\v) \in (0,\infty) \times (0,\pi)^2$.
We point out that the exact behavior of this kernel, as well as its various derivatives with respect to $t,\t$ and $\v$, is hidden behind subtle oscillations. To overcome this problem we will use a convenient method of estimating various kernels defined via $H_t^{\ab}(\t,\v)$ established recently in \cite{NoSj} under a restriction $\ab \ge -1/2$ and then extended in \cite{NoSjSz} to all admissible $\ab > -1$.

Now we introduce the notion of (higher order) derivatives associated with our setting. 
The natural first order derivative $\delta$ emerges from the factorization
$$
\J^{\ab} = \delta^*\delta + \lam_0^{\ab}, 
$$
where
\[
\delta= \frac{d}{d\theta}, \qquad
\delta^{*} = - \frac{d}{d\theta} - (\alpha+1/2)\cot\frac{\t}2+(\beta+1/2)\tan\frac{\t}2;
\]
notice that $\delta^{*}$ is the formal adjoint of $\delta$ in $L^2(d\m_{\ab})$.
The choice of $\delta$ as the first order derivative is motivated by mapping properties of fundamental harmonic analysis operators in the Jacobi framework; see \cite[Remark 2.6]{NoSj} where it is shown that the choice of $\delta^{*}$ would be inappropriate.
On the other hand, the proper choice of higher order derivatives is a much more subtle matter. In the sequel we will consider two choices. Precisely, we can iterate $\delta$ or interlace $\delta$ with $\delta^*$, which leads to $\delta^N$ and 
$$
{D}^N =
    \underbrace{\ldots \delta \delta^* \delta \delta^* \delta}_{N\; \textrm{components}}, \qquad N \ge 0
$$
(by convention, $\delta^0 = D^0 = \id$),
respectively.
We point out that the derivative $\delta^N$ is used in \cite{CNS,NoSj,NoSjSz} whereas ${D}^N$ appears in \cite{CNS,L} and have its roots in the so-called symmetrization procedure proposed in \cite{NoSt}. 

Now we are ready to introduce the central objects of our study in this paper. We define the mixed 
Lusin area integrals as follows
\begin{align*}
S_{M,N}^{\ab}f(\t) 
& = 
\bigg( \int_{\Gamma(\t)} t^{2M+2N-1} \big| \partial_{t}^M \delta^N \mathcal{H}_t^{\ab}f(\eta)
\big|^{2} \, \chi_{ \{ \eta \in (0,\pi)  \} }
\frac{d\mu_{\ab}(\eta) \, dt}
{V_{t}^{\ab}(\t)} \bigg)^{1\slash 2}, \qquad \t \in (0,\pi), \\
\mathcal{S}_{M,N}^{\ab}f(\t) 
& = 
\bigg( \int_{\Gamma(\t)} t^{2M+2N-1} \big| \partial_{t}^M  D^N \mathcal{H}_t^{\ab}f(\eta)
\big|^{2} \, \chi_{ \{ \eta \in (0,\pi)  \} }
\frac{d\mu_{\ab}(\eta) \, dt}
{V_{t}^{\ab}(\t)} \bigg)^{1\slash 2}, \qquad \t \in (0,\pi),
\end{align*}
where $M, N \in \N$ are such that $M + N >0$, 
$\Gamma(\t)$ is the cone with vertex at $\t \in (0,\pi)$,
\[
\Gamma(\t) = (\t,0) + \Gamma, \qquad 
\Gamma = \big\{   (\eta,t) \in \R \times (0,\infty) : |\eta| < t  \big\},
\]
(note that the exact aperture of this cone is insignificant for our developments and that we may replace $\Gamma$ by 
$\widetilde{\Gamma} = \Gamma \cap (-\pi,\pi) \times (0,\infty)$) 
and $V_{t}^{\ab}(\t)$ is the $\mu_{\ab}$ measure of the ball (interval) centered at $\t$ and of radius $t$, restricted to $(0,\pi)$. More precisely,
\[
V_{t}^{\ab}(\t) = \mu_{\ab} \big( B(\t,t) \big), 
\qquad t>0, \quad \t \in (0,\pi).
\]
Observe that the formulas defining $S_{M,N}^{\ab}f$ and $\mathcal{S}_{M,N}^{\ab}f$, understood in a pointwise sense, are valid for any $f \in L^p(wd\mu_{\ab})$, $w \in A_p^{\ab}$, $1 \le p < \infty$; see the comment concerning the smoothness of $\mathcal{H}_t^{\ab}f(\t)$ above.

To obtain the boundedness result for the mixed 
Lusin area integrals in question we will prove that they can be viewed as vector-valued Calder\'on-Zygmund operators in the sense of the space of homogeneous type
$((0,\pi), \mu_{\ab},|\cdot|)$. We will need a slightly more general definition of the standard kernel, or rather standard estimates, than the one used in \cite{NoSj,NoSjSz}.
More precisely, we will allow slightly weaker smoothness estimates as indicated below, see for instance \cite{CaSz,NoStSz,TZS} where the Lusin area integrals were treated in some other orthogonal expansions settings.

Let $\mathbb{B}$ be a Banach space and let $K(\t,\vp)$ be a kernel defined on 
$(0,\pi) \times (0,\pi)\backslash \{ (\t,\vp):\t=\vp \}$ and taking values in $\mathbb{B}$.
We say that $K(\t,\vp)$ is a \emph{standard kernel} in the sense of the space of homogeneous type
$((0,\pi), \mu_{\ab},|\cdot|)$ if it satisfies the so-called \emph{standard estimates}, i{.}e{.},\
the growth estimate
\begin{equation} \label{gr}
\|K(\t,\vp)\|_{\mathbb{B}} \lesssim \frac{1}{\mu_{\ab}(B(\t,|\t-\vp|))}, 
\qquad \t \ne \vp,
\end{equation}
and the smoothness estimates
\begin{align}
\| K(\t,\vp)-K(\t',\vp)\|_{\mathbb{B}} & \lesssim 
\left( \frac{|\t-\t'|}{|\t-\vp|} \right)^\gamma \,
 \frac{1}{\mu_{\ab}(B(\t,|\t-\vp|))},
\qquad |\t-\vp|>2|\t-\t'|, \label{sm1}\\
\| K(\t,\vp)-K(\t,\vp')\|_{\mathbb{B}} & \lesssim 
\left( \frac{|\vp-\vp'|}{|\t-\vp|}\right)^\gamma \,
 \frac{1}{\mu_{\ab}(B(\t,|\t-\vp|))},
\qquad |\t-\vp|>2|\vp-\vp'|, \label{sm2}
\end{align}
for some fixed $\gamma > 0$. 
Notice that the right-hand side of \eqref{gr} is always larger than the constant $1/\mu_{\ab}(0,\pi)$, which will be used frequently in the sequel without further mention. 
Further, notice that the bounds \eqref{sm1} and \eqref{sm2} imply analogous estimates 
with any $0<\gamma'<\gamma$ instead of $\gamma$.
Moreover, observe that in these formulas the ball (interval) $B(\t,|\t-\vp|)$ can be replaced by $B(\vp,|\vp-\t|)$, 
in view of the doubling property of $\mu_{\ab}$.

A linear operator $T$ assigning to each $f\in L^2(d\mu_{\ab})$
a strongly measurable $\mathbb{B}$-valued function $Tf$ on $(0,\pi)$. 
Then $T$ is said to be a (vector-valued)
Calder\'on-Zygmund operator in the sense of the space $((0,\pi),\mu_{\ab},|\cdot|)$ associated with
$\mathbb{B}$ if
\begin{itemize}
\item[(A)] $T$ is bounded from $L^2(d\mu_{\ab})$ to $L^2_{\mathbb{B}}(d\mu_{\ab})$,
\item[(B)] there exists a standard $\mathbb{B}$-valued kernel $K(\t,\vp)$ such that
$$
Tf(\t) = \int_0^{\pi} K(\t,\vp) f(\vp)\, d\mu_{\ab}(\vp), \qquad \textrm{a.a.}\; \t \notin \support f,
$$
for every $f \in L^{\infty}((0,\pi))$.
\end{itemize}
Here integration of $\mathbb{B}$-valued functions is understood in Bochner's sense, and
$L^2_{\mathbb{B}}(d\mu_{\ab})$ is the Bochner-Lebesgue space of all $\mathbb{B}$-valued 
$d\mu_{\ab}$-square integrable functions on $(0,\pi)$.

It is well known that a large part of the classical theory of Calder\'on-Zygmund operators remains valid,
with appropriate adjustments, when the underlying space is of homogeneous type and the associated kernels
are vector-valued, see for instance \cite{RRT,RT}. In particular, if $T$ is a Calder\'on-Zygmund
operator in the sense of $((0,\pi),\mu_{\ab},|\cdot|)$ associated with a Banach space $\mathbb{B}$,
then its mapping properties in weighted $L^p$ spaces follow from the general theory. 

Obviously, the mixed Lusin area integrals $S^{\ab}_{M,N}$ and $\mathcal{S}^{\ab}_{M,N}$ are nonlinear. However, they can be written as 
\begin{align*}
S_{M,N}^{\ab}f(\t) 
&= 
\big\| \partial_{t}^M \delta^N \mathcal{H}_t^{\ab}f(\t + \eta)
\,
\chi_{ \{ \t + \eta \in (0,\pi) \} }  \sqrt{\ot} 
  \big\|_{L^2(\Gamma, t^{2M+2N-1}d\eta dt)},
\qquad \t \in (0,\pi), \\
\mathcal{S}_{M,N}^{\ab}f(\t) 
&= 
\big\|  \partial_{t}^M D^N \mathcal{H}_t^{\ab}f(\t + \eta)
\,
\chi_{ \{ \t + \eta \in (0,\pi) \} }  \sqrt{\ot} 
  \big\|_{L^2(\Gamma, t^{2M+2N-1}d\eta dt)},
\qquad \t \in (0,\pi),
\end{align*}
where the function $\ottt$ is given by
\[
\ot = \frac{ \big( \sin\frac{\t+\eta}2 \big)^{2\a+1} \, 
\big( \cos\frac{\t+\eta}2 \big)^{2\b+1}  }{V_t^{\ab}(\t)}, 
\qquad 
\eta \in \R, \quad t>0,
\quad \t,\t+\eta \in (0,\pi).
\]
This, in turn, shows that these operators can be viewed as vector-valued linear operators taking values in $\mathbb{B} = L^2(\Gamma, t^{2M+2N-1}d\eta dt)$.
Further, note that the formal computation suggests that the kernels associated with $S^{\ab}_{M,N}$ and $\mathcal{S}^{\ab}_{M,N}$ are given by
\begin{align*}
S^{\ab}_{M,N}(\t,\vp) = 
\big\{ S^{\ab}_{M,N,\eta,t}(\t,\vp)
\big\}_{(\eta,t) \in \Gamma}, \qquad
\mathcal{S}^{\ab}_{M,N}(\t,\vp) = 
\big\{ \mathcal{S}^{\ab}_{M,N,\eta,t}(\t,\vp)
\big\}_{(\eta,t) \in \Gamma},
\end{align*}
where $M,N \in \N$ are such that $M+N>0$, and
\begin{align*}
S^{\ab}_{M,N,\eta,t}(\t,\vp) & =
\partial_t^M \delta_\psi^N H_t^{\ab}(\psi,\vp) \big|_{\psi = \t + \eta} 
\, \chi_{ \{ \t + \eta \in (0,\pi) \} }
\, \sqrt{\ot}, \\
\mathcal{S}^{\ab}_{M,N,\eta,t}(\t,\vp) & =
\partial_t^M D_\psi^N H_t^{\ab}(\psi,\vp) \big|_{\psi = \t + \eta}
\, \chi_{ \{ \t + \eta \in (0,\pi) \} }
\, \sqrt{\ot}.
\end{align*}

The main result of the paper reads as follows.

\begin{thm}\label{thm:CZ}
Let $\ab > -1$ and $M,N \in \N$ be such that $M+N > 0$. Then the 
mixed Lusin area integrals $S^{\ab}_{M,N}$ and $\mathcal{S}^{\ab}_{M,N}$ can be viewed as vector-valued Calder\'on-Zygmund operators in the sense of $((0,\pi),\mu_{\ab},|\cdot|)$ associated with the Banach spaces
$\mathbb{B} = L^2(\Gamma, t^{2M+2N-1}d\eta dt)$.
\end{thm}

The proof of Theorem \ref{thm:CZ} splits naturally into the following two results.

\begin{pro}\label{pro:L2ass}
Let $\ab > -1$ and $M,N \in \N$ be such that $M+N > 0$. Then the 
mixed Lusin area integrals $S^{\ab}_{M,N}$ and $\mathcal{S}^{\ab}_{M,N}$ are bounded on $L^2(d\mu_{\ab})$. Moreover, $S^{\ab}_{M,N}$ and $\mathcal{S}^{\ab}_{M,N}$, 
viewed as vector-valued operators, are associated in the Calder\'on-Zygmund theory sense with the kernels $S^{\ab}_{M,N}(\t,\vp)$ and $\mathcal{S}^{\ab}_{M,N}(\t,\vp)$, respectively.
\end{pro}

\begin{proof}
By Lemma \ref{lem:Omega} below we see that proving the $L^2(d\mu_{\ab})$-boundedness of $S^{\ab}_{M,N}$ and $\mathcal{S}^{\ab}_{M,N}$ is equivalent to showing an analogous property for the related $g$-functions $g^{\ab}_{M,N}$ and $\big( G^{\ab}_{M,N} \big)^+$ investigated in \cite{NoSjSz} and \cite{L}, respectively. 
This, however, was recently established, even on weighted $L^p$, $1 < p < \infty$, spaces, in \cite[Corollary 5.2]{NoSjSz} and \cite[Theorem 3.2 and Proposition 3.9]{L}, respectively. 

To prove the kernel associations one can proceed as in \cite[Proposition 2.5 on pp.\,1528-1529]{TZS}, where similar fact was proved for the first order Lusin area integrals in the Laguerre-Dunkl context. The crucial ingredients needed in the reasoning are the just explained $L^2(d\mu_{\ab})$-boundedness of the operators in question and the standard estimates for the kernels $S^{\ab}_{M,N}(\t,\vp)$ and $\mathcal{S}^{\ab}_{M,N}(\t,\vp)$. The latter fact is justified in Theorem \ref{thm:kerest} below. The details are fairly standard and thus left to the reader.
\end{proof}

\begin{thm}\label{thm:kerest}
Assume that $\ab > -1$, and let $M,N \in \N$ be such that $M+N > 0$. Then the kernels $S^{\ab}_{M,N}(\t,\vp)$ and $\mathcal{S}^{\ab}_{M,N}(\t,\vp)$ satisfy the standard estimates with the Banach spaces $\mathbb{B} = L^2(\Gamma, t^{2M+2N-1}d\eta dt)$ and with any 
$\gamma \in (0,1/2]$ satisfying $\gamma < \a \wedge \b + 1$.
\end{thm}

The proof of Theorem \ref{thm:kerest}, which is the most technical part of the paper, is located in Section~\ref{sec:kerest}.

An important consequence of Theorem \ref{thm:CZ} is the following result.

\begin{cor}
Let $\ab > -1$ and $M,N \in \N$ be such that $M+N > 0$. Then the 
mixed Lusin area integrals $S^{\ab}_{M,N}$ and $\mathcal{S}^{\ab}_{M,N}$ are bounded on $L^p(wd\mu_{\ab})$, $w \in A_p^{\ab}$, $1 < p < \infty$, and from $L^1(wd\mu_{\ab})$ to $L^{1, \infty} (wd\mu_{\ab})$, $w \in A_1^{\ab}$.
\end{cor}

\begin{proof}
The fact that $S^{\ab}_{M,N}$ and $\mathcal{S}^{\ab}_{M,N}$ extend to bounded operators on $L^p(wd\mu_{\ab})$, $w \in A_p^{\ab}$, $1 < p < \infty$, and from $L^1(wd\mu_{\ab})$ to $L^{1, \infty} (wd\mu_{\ab})$, $w \in A_1^{\ab}$ follows from the (vector-valued) Calder\'on-Zygmund theory. Therefore it suffices to justify that these extensions coincide with the original definitions. The latter, however, can be done by using arguments similar to those mentioned in the proof of \cite[Corollary 2.5]{NoSj}.
Further details are left to the reader. 
\end{proof}

We note that some mapping properties for the first order vertical Lusin area integrals, i.e.\ with $M=1$ and $N = 0$, follow from results established in a general framework of spaces of homogeneous type. Using \cite[Theorem 1.1]{GX} we obtain, in particular, weighted weak type $(1,1)$ estimate for  $S^{\ab}_{1,0} = \mathcal{S}^{\ab}_{1,0}$ with all $A_1^{\ab}$ weights admitted. The assumptions imposed in \cite{GX} are indeed satisfied, since the Jacobi-heat kernel of the Jacobi-heat semigroup $\big\{ \exp(-t \mathcal{J}^{\ab}) \big\}_{t > 0}$ possesses the so-called Gaussian bound. The latter was recently established independently by Coulhon, Kerkyacharian, Petrushev in \cite[Theorem 7.2]{CKP} and by Nowak, Sj\"ogren in 
\cite[Theorem A]{NoSj2}.

\section{Kernel estimates} \label{sec:kerest}
This section is devoted to proving standard estimates for the kernels $S^{\ab}_{M,N}(\t,\vp)$ and $\mathcal{S}^{\ab}_{M,N}(\t,\vp)$ related to the Banach spaces $\mathbb{B} = L^2(\Gamma, t^{2M+2N-1}d\eta dt)$ and asserted in Theorem \ref{thm:kerest}.
We point out that a prominent role in our proof is played by the method of proving kernel estimates established in \cite{NoSj} under the restriction $\ab \ge -1/2$ and then generalized in \cite{NoSjSz} to all admissible type parameters $\ab >-1$. 
In those papers a convenient integral representation for the Jacobi-Poisson kernel was established, see \cite[(1) and Proposition 2.3]{NoSjSz}. This allowed the authors to elaborate a technique of estimating various kernels expressible via $H_{t}^{\ab}(\t,\vp)$ and in consequence to show that several fundamental harmonic analysis operators in the Jacobi context are (vector-valued) Calder\'on-Zygmund operators. 
However, the Lusin area integrals have more complex structure than the operators investigated in \cite{NoSj,NoSjSz} and thus to treat them we need to establish further generalization of this interesting method. To achieve this, we will adapt some intuitions and ideas from the contexts of the Dunkl harmonic oscillator \cite{TZS} and the Dunkl Laplacian \cite{CaSz}, where analogous techniques were developed. 
We emphasize that the analysis related to the restricted range of $\ab \ge -1/2$ is much simpler than the one concerning all $\ab > -1$. This is caused by the fact that for $\ab \ge -1/2$ the above mentioned integral representation for $H_{t}^{\ab}(\t,\vp)$ is less complicated. However, in the sequel we would like to treat all $\ab$ in a unified way.

To prove the kernel estimates stated in Theorem \ref{thm:kerest} we will need some preparatory results, which are gathered below. Some of them were obtained in the previous papers \cite{NoSj,NoSjSz}, but we recall them here for the sake of reader's convenience.
To state them we shall use the same notation as in \cite{NoSjSz}. 
For $\a > -1/2$, let $d\Pi_\a$ be the probability measure on the interval $[-1,1]$ defined by
\[
d\Pi_\a(u) = \frac{\Gamma(\a + 1)}{\sqrt{\pi} \Gamma(\a + 1/2)} 
(1-u^2)^{\a-1/2} \, du,
\]
and in the limit case $d\Pi_{-1/2}$ is the sum of point masses at $-1$ and $1$ divided by $2$.
Further, let
$$
d\abpiK
= d\Pi_{(\a+1)K - (1-K)/2}
= \begin{cases}
        d\Pi_{-1\slash 2}, & K=0 \\
        d\Pi_{\a + 1}, &  K=1
    \end{cases}, \qquad \a >-1,
$$
and put
\[
q(\t,\vp,u,v) = 1 - u \st \svp - v \ct \cvp, 
\qquad \t,\vp \in (0,\pi), \quad u,v \in [-1,1].
\]
Furthermore, for  $W, s \in \R$ fixed, we consider a function $\Upsilon_{W,s}^{\ab}(t,\t,\vp)$ defined on 
$(0,\pi] \times (0,\pi) \times (0,\pi)$ as follows.

\begin{itemize}
\item[(i)]
For $\a,\b \ge -1\slash 2$,
\begin{align*}
\Upsilon^{\ab}_{W,s}(t,\t,\vp) := 
\iint
\frac{\pia \, \pib }{(t^2 + q(\t,\vp,u,v))^{ \a+\b+3\slash 2 + W/4 + s\slash 2  }}.
\end{align*}
\item[(ii)]
For $-1 < \a < -1\slash 2 \le \b$,
\begin{align*}
\Upsilon^{\ab}_{W,s}(t,\t,\vp) := & \,
1 + \sum_{K=0,1} \sum_{k=0,1,2}
\bigg( \st + \svp \bigg)^{Kk} \\
& \qquad \times \iint
\frac{\piK \, \pib }{(t^2 + q(\t,\vp,u,v))^{ \a+\b+3\slash 2 + W/4 +Kk\slash 2 + s\slash 2  }}.
\end{align*}
\item[(iii)]
For $-1 < \b < -1\slash 2 \le \a$,
\begin{align*}
\Upsilon^{\ab}_{W,s}(t,\t,\vp) := & \,
1 + \sum_{R=0,1} \sum_{r=0,1,2}
\bigg( \ct + \cvp \bigg)^{Rr} \\
& \qquad \times \iint
\frac{\pia \, \piR }{(t^2 + q(\t,\vp,u,v))^{ \a+\b+3\slash 2 + W/4 +Rr\slash 2 + s\slash 2  }}.
\end{align*}
\item[(iv)] 
For $-1 < \a,\b < -1\slash 2$,
\begin{align*}
\Upsilon^{\ab}_{W,s}(t,\t,\vp) := & \,
1 + \sum_{K,R=0,1} \sum_{k,r=0,1,2}
\bigg( \st+\svp \bigg)^{Kk} \bigg( \ct + \cvp \bigg)^{Rr} \\
& \qquad \times \iint
\frac{\piK \, \piR }{(t^2 + q(\t,\vp,u,v))^{ \a+\b+3\slash 2 + W/4 +Kk\slash 2 + Rr\slash 2 + s\slash 2  }}.
\end{align*}
\end{itemize}
Notice that for any $\tau \in \R$ we have $\Upsilon^{\ab}_{W,s}(t,\t,\vp) = \Upsilon^{\ab}_{W-2\tau,s+\tau}(t,\t,\vp)$, which will frequently be used in the sequel.

Now we are ready to state the following lemma, which was partially proved in \cite[Corollary 3.5]{NoSjSz}.

\begin{lem} \label{lem:Ht1}
Let $M,N \in \N$ and $L,P \in \{0,1\}$ be fixed. Then,
\[
\big| 
	 \partial_\vp^L \partial_\t^P \partial_t^M \delta_\t^N
H_{t}^{\ab}(\t,\vp)  
\big| 
+
\big| 
	 \partial_\vp^L \partial_\t^P \partial_t^M D_\t^N
H_{t}^{\ab}(\t,\vp)  
\big| 
\lesssim 
\Upsilon^{\ab}_{2M+2N,L+P}(t,\t,\vp)
\]
uniformly in $t\in (0,\pi]$ and $\t,\vp \in (0,\pi)$.
\end{lem}

\begin{proof}
We first note that the estimates obtained in \cite[Corollary 3.5]{NoSjSz} actually hold uniformly in $t\in (0,T_0]$, where $T_0 > 0$ is an arbitrary but fixed constant. Therefore, the estimate for the first term in question is just this strengthened version of the above mentioned result since 
$\delta^N =  \partial^N$. We now focus on the remaining term.

Notice that the Jacobi-Poisson semigroup $\mathcal{H}_t^{\ab} = \exp(-t \sqrt{\J^{\ab}})$, $t>0$, satisfies the differential equation
$\partial_t^2 \mathcal{H}_t^{\ab} f(\t) = \J^{\ab} \mathcal{H}_t^{\ab} f(\t)$. This forces that the Jacobi-Poisson kernel $H_t^{\ab}(\t,\v)$ also satisfies this equation with respect to $\t$, which, together with the identity 
$\J^{\ab} = D^2 + \lam_0^{\ab}$, leads to the equation
\[
D_{\t}^2 H_t^{\ab}(\t,\v) = (\partial_t^2 - \lam_0^{\ab}) H_t^{\ab}(\t,\v), 
\qquad t > 0, \quad \t,\vp \in (0,\pi).
\]  
Iterating this we infer that for $t > 0$ and $\t,\vp \in (0,\pi)$ we get
\begin{align} \label{iden1}
\partial_\vp^L \partial_\t^P \partial_t^M D_\t^N
H_{t}^{\ab}(\t,\vp)  
=
\sum_{n=0}^{\lfloor N/2 \rfloor} c_n (\lam_0^{\ab})^{\lfloor N/2 \rfloor - n}
\partial_\vp^L \partial_\t^{P+ \overline{N} } \partial_t^{M + 2n} 
H_{t}^{\ab}(\t,\vp), 
\end{align}
where $\overline{N} = \chi_{\{N \; \textrm{is odd}\}}$ and $c_n$ are some constants.
Now, using the already justified estimate for the first term in question and the fact that for any $-\infty < W \le W' < \infty$, $s \in \R$ fixed we have
\begin{align*}
\Upsilon^{\ab}_{W,s}(t,\t,\vp) \lesssim
\Upsilon^{\ab}_{W',s}(t,\t,\vp),
\qquad t \in (0, \pi], \quad \t,\vp \in (0,\pi)
\end{align*}
(this easily follows from the boundedness of $q(\t,\vp,u,v)$), we obtain the desired conclusion for the second term in question.
\end{proof}

We need the following generalization of the above lemma.

\begin{lem}\label{lem:Ht1eta}
Let $M,N \in \N$ and $L,P \in \{0,1\}$ be fixed. Then,
\[
\big| 
	 \partial_\vp^L \partial_\psi^P \partial_t^M \delta_\psi^N
H_{t}^{\ab}(\psi,\vp) \big|_{\psi = \t + \eta} 
\big| 
+
\big| 
	 \partial_\vp^L \partial_\psi^P \partial_t^M D_\psi^N
H_{t}^{\ab}(\psi,\vp) \big|_{\psi = \t + \eta} 
\big| 
\lesssim 
\Upsilon^{\ab}_{2M+2N,L+P}(t,\t,\vp)
\]
uniformly in $(\eta,t) \in \Gamma$, $t\in (0,\pi]$ and 
$\t,\vp \in (0,\pi)$ satisfying $\t + \eta \in (0,\pi)$.
\end{lem}

To prove Lemma \ref{lem:Ht1eta} we shall use the following result.

\begin{lem}\label{lem:qeta}
For all $(\eta,t) \in \Gamma$, and $\t,\vp \in (0,\pi)$ with $\t + \eta \in (0,\pi)$, and all $u,v \in [-1,1]$, we have
\[
t^2 + q(\t + \eta,\vp,u,v) \simeq
t^2 + q(\t,\vp,u,v).
\]
\end{lem} 

\begin{proof}
We first show the estimate $(\lesssim)$. To prove this, it suffices to verify that 
\begin{equation*}
q(\t + \eta,\vp,u,v) \lesssim
t^2 + q(\t,\vp,u,v), 
\qquad (\eta,t) \in \Gamma, \quad \t, \vp, \t + \eta \in (0,\pi),
\quad u,v \in[-1,1]. 
\end{equation*}
Observe that the following comparability holds, see \cite[(23)]{NoSj},
\[
q(\t,\vp,u,v) \simeq
(\t-\vp)^2 + (1-u) \t \vp + (1-v) (\pi-\t) (\pi-\vp), \qquad \t, \vp \in (0,\pi), \quad u,v \in[-1,1].
\]
Further, since
\begin{align*}
(\t + \eta - \vp )^2 & \lesssim  (\t - \vp)^2 + \eta^2 \le (\t - \vp)^2 + t^2, \\
(1-u) (\t + \eta) \vp & \le  (1-u) \t \vp + (1-u) t \vp, \\
(1-v) (\pi-\t - \eta) (\pi-\vp) & \le  (1-v) (\pi-\t) (\pi-\vp) + (1-v) t (\pi-\vp),
\end{align*}
it is enough to check that
\begin{align*}
t \vp & \lesssim t^2 + (\t - \vp)^2 + \t \vp, 
\qquad t > 0, \quad \t, \vp \in (0,\pi)\\
t (\pi-\vp) & \lesssim t^2 + (\t - \vp)^2 + (\pi-\t) (\pi-\vp), 
\qquad t > 0, \quad \t, \vp \in (0,\pi).
\end{align*}
This, in turn, follows from the relation 
$(\t - \vp)^2 + \t \vp \simeq \t^2 + \vp^2$ for $\t, \vp \in (0,\pi)$, and by a reflection in $\pi/2$ reason.
Therefore we proved $(\lesssim)$. 

Now, the bound $(\gtrsim)$ is a straightforward consequence of the first one. 
Precisely,
\[
t^2 + q(\t,\vp,u,v)
= 
t^2 + q((\t + \eta) + (-\eta),\vp,u,v)
\lesssim
t^2 + q(\t + \eta,\vp,u,v)
\]
since $(-\eta,t) \in \Gamma$ and $\t, \t + \eta \in (0,\pi)$.
\end{proof}

\begin{proof}[Proof of Lemma \ref{lem:Ht1eta}.]
By Lemma \ref{lem:Ht1} it suffices to show that for each $W, s \in \R$  fixed, we have
\[
\Upsilon_{W,s}^{\ab} (t, \t + \eta, \vp)
\lesssim
\Upsilon_{W,s}^{\ab} (t, \t, \vp), \qquad (\eta,t) \in \Gamma, 
\quad \t,\vp, \t + \eta \in  (0,\pi).
\]
We prove this estimate when $-1<\ab <-1/2$; the proofs of the remaining cases are similar and even simpler, and hence are omitted. To proceed, notice that for any $\kappa \ge 0$ fixed we have
\begin{align*}
\Big( \sin \frac{\t + \eta}{2} + \sin \frac{\vp}{2} \Big)^\kappa
	& \lesssim
\Big( \sin \frac{\t}{2} + \sin \frac{\vp}{2} \Big)^\kappa
+ t^\kappa, \qquad (\eta,t) \in  \Gamma, \quad \t,\vp, \t + \eta \in  (0,\pi), \\
\Big( \cos \frac{\t + \eta}{2} + \cos \frac{\vp}{2} \Big)^\kappa
	& \lesssim
\Big( \cos \frac{\t}{2} + \cos \frac{\vp}{2} \Big)^\kappa
+ t^\kappa, \qquad (\eta,t) \in  \Gamma, \quad \t,\vp, \t + \eta \in  (0,\pi).
\end{align*}
These can easily be justified with the aid of the relations
\begin{equation}\label{compsin}
\sin \frac{x}2 \simeq x, \qquad \cos \frac{x}2 \simeq \pi - x, \quad x \in (0,\pi),
\end{equation}
and the fact that $|\eta| < t$.
Using these bounds, Lemma \ref{lem:qeta} and then the obvious inequality 
$t \le (t^2 + q(\t,\vp,u,v))^{1/2}$, we get
\begin{align*}
\Upsilon_{W,s}^{\ab} (t, \t + \eta, \vp) 
	& \lesssim
1 + \sum_{K,R=0,1} \sum_{k,r=0,1,2} \sum_{\varepsilon_1, \varepsilon_2 = 0,1}
\bigg( \st+\svp \bigg)^{Kk\varepsilon_1} 
\bigg( \ct + \cvp \bigg)^{Rr\varepsilon_2} \\
& \qquad \quad \times  t^{Kk(1-\varepsilon_1) + Rr(1-\varepsilon_2)}  
\iint
\frac{\piK \, \piR }{(t^2 + q(\t,\vp,u,v))^{ \a+\b+3\slash 2 + W/4 +Kk\slash 2 + Rr\slash 2 + s\slash 2  }} \\
	& \lesssim
1 + \sum_{K,R=0,1} \sum_{k,r=0,1,2} \sum_{\varepsilon_1, \varepsilon_2 = 0,1}
\bigg( \st+\svp \bigg)^{Kk\varepsilon_1} \bigg( \ct + \cvp \bigg)^{Rr\varepsilon_2} \\
& \qquad \quad \times \iint
\frac{\piK \, \piR }{(t^2 + q(\t,\vp,u,v))^{ \a+\b+3\slash 2 + W/4 +Kk\varepsilon_1\slash 2 + Rr\varepsilon_2\slash 2 + s\slash 2  }}.
\end{align*}
Since $k\varepsilon_1, r\varepsilon_2 \in \{ 0,1,2 \}$, the last quantity above is comparable to $\Upsilon_{W,s}^{\ab} (t, \t, \vp) $ and the conclusion follows.
\end{proof}

The next result is a slightly modified special case of \cite[Lemma 3.7]{NoSjSz} (note that the norm estimate in \cite[Lemma 3.7, p.\,201, l.\,2]{NoSjSz} is still valid if we replace $L^p((0,1),t^{W-1}dt)$ appearing there by $L^p((0,T_0),t^{W-1}dt)$ with any $T_0 > 0$ fixed), which played a crucial role in showing standard estimates for various kernels investigated in the just mentioned paper. This lemma provides an important connection between estimates emerging from Lemma \ref{lem:Ht1eta} and the standard estimates related to the space of homogeneous type $( (0,\pi), d\mu_{\ab}, |\cdot|)$.

\begin{lem}[{\cite[Lemma 3.7]{NoSjSz}}] \label{lem:finbridgep=2}
Let $W \ge 1$ and $s\ge 0$ be fixed. Then the estimate
\[
\| \Upsilon^{\ab}_{W,s}(t,\t,\vp) \|_{ L^2((0,\pi),t^{W-1}dt) } 
\lesssim
\frac{1}{|\t-\vp|^s} \; \frac{1}{\mu_{\ab}(B(\t,|\t-\vp|))}
\]
holds uniformly in $\t,\vp \in (0,\pi)$, $\t\ne\vp$.
\end{lem}

Note that for any $\ab > -1$ fixed, the $\mu_{\ab}$ measure of the ball (interval) $B(\t,r)$ can be described as follows, cf. \cite[Lemma 4.2]{NoSj},
\begin{equation}\label{ball}
\mu_{\ab}(B(\t,r)) \simeq 
\begin{cases}
        r(\t+r)^{2\a + 1} (\pi-\t+r)^{2\b + 1}, & r \in (0,\pi) \\
        1, & r \ge \pi
    \end{cases}, \qquad \t \in (0,\pi).
\end{equation}

The next lemma is a long-time counterpart of Lemma \ref{lem:finbridgep=2} and was partially proved in \cite[Corollary 3.9]{NoSjSz}.

\begin{lem} \label{cor:tLp=2}
Let $\ab > -1$, $M,N \in \N$ be such that $M+N>0$, and let $L,P \in \{0,1\}$ and $W \ge 1$ be fixed.
Then
$$
\Big\|\sup_{\t,\vp \in (0,\pi) } 
\Big(
\big| 
	 \partial_\vp^L \partial_\t^P \partial_t^M \delta_\t^N
H_{t}^{\ab}(\t,\vp)  
\big| 
+
\big| 
	 \partial_\vp^L \partial_\t^P \partial_t^M D_\t^N
H_{t}^{\ab}(\t,\vp)  
\big| 
\Big)
	\Big\|_{L^2((\pi,\infty),t^{W-1}dt)} < \infty.
$$
\end{lem}

\begin{proof}
Since $\delta^N = \partial^N$, the norm finiteness of the first expression in question is a direct consequence of \cite[Corollary 3.9]{NoSjSz}. The conclusion for the second term follows by combining the identity \eqref{iden1} with \cite[Corollary 3.9]{NoSjSz}.
\end{proof}

The next three lemmas (Lemmas \ref{lem:Omega}--\ref{lem:Omegaprime}) allow us to control certain expressions involving the function $\ottt$ appearing in the definitions of the kernels 
$S^{\ab}_{M,N}(\t,\vp)$ and $\mathcal{S}^{\ab}_{M,N}(\t,\vp)$.

\begin{lem}\label{lem:Omega}
Assume that $\ab > -1$. Then
\[
\int_{|\eta|<t} \chi_{ \{ \t + \eta \in (0,\pi) \} } \ot \, d\eta =1,
\qquad t>0, \quad \t \in (0,\pi).
\]
\end{lem}

\begin{proof}
Simple exercise.
\end{proof}

\begin{lem}\label{lem:Omegadiff}
Assume that $\ab > -1$. 
Then there exists $\gamma=\gamma(\a,\b) \in (0,1/2]$ such that
\begin{align*}
\int_{|\eta|<t} \chi_{ \{ \t + \eta, \t' + \eta  \in (0,\pi) \} } 
\Big| \sqrt{\ot} - \sqrt{\ott} \Big|^2 \, d\eta 
\lesssim
\begin{cases}
        \Big( \frac{|\t -\t'|}{t} \Big)^{2\gamma}, & t \in (0,\pi] \\
        |\t -\t'|^{2\gamma}, &  t \ge \pi
    \end{cases}, 
\end{align*}
uniformly in $t > 0$ and $\t, \t' \in (0,\pi)$.
Moreover, the estimate holds with any $\gamma \in (0,1/2]$ satisfying 
$\gamma < \a \wedge \b + 1$.
\end{lem}

In the proof of this lemma we will use the following estimate. For a fixed $\xi \ge 1$ we have
\begin{equation}\label{estxyxi}
|x-y|^{\xi} \lesssim |x^{\xi}-y^{\xi}|, \qquad x,y \ge 0.
\end{equation}

\begin{proof}[Proof of Lemma \ref{lem:Omegadiff}]
Since the function $\ot$ stabilizes for $t \ge \pi$ and the constraint $\t, \t + \eta \in (0,\pi)$ forces $|\eta| < \pi$,
we may assume that $0 < t \le \pi$.
Further, from now on we restrict our attention to $|\t - \t'| \le t$. Otherwise, an application of Lemma \ref{lem:Omega} shows that the left-hand side in question is controlled by a constant and the conclusion trivially follows.

Using the estimate \eqref{estxyxi} with $\xi=2$ we obtain
\begin{align*}
& \chi_{ \{ \t + \eta, \t' + \eta  \in (0,\pi) \} } 
\Big| \sqrt{\ot} - \sqrt{\ott} \Big|^2 \\
	& \quad \lesssim
\chi_{ \{ - \t \wedge \t' < \eta < \pi - \t \vee \t' \} } 
\big| \ot - \ott \big| \\
	& \quad \le 
\chi_{ \{ - \t \wedge \t' < \eta < \pi - \t \vee \t' \} } 
\frac{\Big| \big( \ste \big)^{2\a +1} \big( \cte \big)^{2\b +1}  
- \big( \stpe \big)^{2\a +1} \big( \ctpe \big)^{2\b +1} \Big|}{ V_t^{\ab}(\t') }  \\
	& \quad \qquad +
\chi_{ \{ \t + \eta  \in (0,\pi) \} }
\ot
\frac{
\big| V_t^{\ab}(\t) - V_t^{\ab}(\t') \big|}{ V_t^{\ab}(\t') }  \\
		& \quad \equiv
I_1(\t,\t',\eta,t)  + I_2(\t,\t',\eta,t). 
\end{align*}
Therefore we have reduced the proof of the lemma to showing that for $j=1,2$ we have
\begin{equation}\label{Iest}
\int_{|\eta| < t} I_j(\t,\t',\eta,t) \, d\eta 
\lesssim
\bigg( \frac{|\t -\t'|}{t} \bigg)^{2\gamma}, \qquad t\in (0,\pi], 
\quad \t,\t' \in (0,\pi), \quad |\t- \t'| \le t,
\end{equation} 
for each fixed $\gamma \in (0,1/2]$ satisfying $\gamma < \a \wedge \b + 1$.

We will treat $I_1$ and $I_2$ separately.
Let 
$f(\t)=\big( \ste \big)^{2\a +1} \big( \cte \big)^{2\b +1}$. Applying the estimate \eqref{estxyxi} specified to $\xi=1/(2\gamma)$ and then the Mean Value Theorem we obtain
\begin{align*}
|f(\t) - f(\t')|
\lesssim
\big| f(\t)^{1/(2\gamma)} - f(\t')^{1/(2\gamma)} \big|^{2\gamma}
\lesssim
|\t - \t'|^{2\gamma} f(\widetilde{\t})^{ 1 - 2\gamma } |f'(\widetilde{\t})|^{2\gamma},
\end{align*}
where $\widetilde{\t}$ is a convex combination of $\t$ and $\t'$, which depends also on $\eta$. An elementary computation gives
\begin{align*}
|f'(\t)| & \lesssim
\Big( \ste \Big)^{2\a} \Big( \cte \Big)^{2\b} \\
	& \simeq
\Big( \ste \Big)^{2\a + 1/(2\gamma)} \Big( \cte \Big)^{2\b}
+ 
\Big( \ste \Big)^{2\a} \Big( \cte \Big)^{2\b + 1/(2\gamma)}.
\end{align*}
Using this 
we get
\begin{align*}
I_1(\t,\t',\eta,t)
	 & \lesssim
|\t - \t'|^{2\gamma} 
\chi_{ \{ - \t \wedge \t' < \eta < \pi - \t \vee \t' \} } 
\frac{\big(\sin \frac{\widetilde{\t} + \eta}{2}\big)^{2\a + 2 - 2\gamma} 
\big(\cos \frac{\widetilde{\t} + \eta}{2}\big)^{2\b + 1 - 2\gamma} }{V_t^{\ab}(\t')} \\ 
	 & \qquad +
|\t - \t'|^{2\gamma} 
\chi_{ \{ - \t \wedge \t' < \eta < \pi - \t \vee \t' \} } 
\frac{\big(\sin \frac{\widetilde{\t} + \eta}{2}\big)^{2\a + 1 - 2\gamma} 
\big(\cos \frac{\widetilde{\t} + \eta}{2}\big)^{2\b + 2 - 2\gamma} }{V_t^{\ab}(\t')} \\ 
	 & \equiv
J_1(\t,\t',\eta,t) + J_2(\t,\t',\eta,t).
\end{align*}
We first deal with $J_1$. Using the 
fact that $2\a + 2 - 2\gamma > 0$, $|\eta|<t$ and
the estimates \eqref{compsin}, \eqref{ball}, we obtain
\[
J_1(\t,\t',\eta,t)
\lesssim
|\t - \t'|^{2\gamma} 
\chi_{ \{ - \t \wedge \t' < \eta < \pi - \t \vee \t' \} }
\frac{(\t \vee \t' + t)^{2\a + 2 - 2\gamma} 
\big( \cos \frac{\t_* + \eta}{2}\big)^{2\b + 1 - 2\gamma} }
{t (\t' +t)^{2\a + 1} (\pi - \t' +t)^{2\b + 1}},
\]
where
\[
\t_* = 
\left\{ \begin{array}{ll}
\t \wedge \t', & 2\b + 1 - 2\gamma \ge 0 \\
\t \vee \t', & 2\b + 1 - 2\gamma < 0
\end{array} \right..
\]
Using the relations 
\begin{equation}\label{comp1}
\t' + t \simeq \t + t \lesssim 1, \qquad \pi - \t' + t \simeq \pi - \t + t \lesssim 1, 
\end{equation} 
(which are valid because $|\t - \t'| \le t$ and $t \in (0,\pi]$), the fact that $1-2\gamma \ge 0$ and the above bound for $J_1$ we arrive at 
\[
J_1(\t,\t',\eta,t)
\lesssim
|\t - \t'|^{2\gamma} 
\chi_{ \{ - \t_* < \eta < \pi - \t_* \} }
\frac{
\big( \cos \frac{\t_* + \eta}{2}\big)^{2\b + 1 - 2\gamma} }
{t  (\pi - \t' +t)^{2\b + 1}},
\] 
and consequently
\begin{align*}
\int_{|\eta|<t} J_1(\t,\t',\eta,t) \, d\eta
\lesssim
\frac{ |\t - \t'|^{2\gamma} }
{t  (\pi - \t' +t)^{2\b + 1}} 
\mu_{-1/2,\b - \gamma} (B(\t_*,t)).
\end{align*}
Now an application of \eqref{ball} and once again \eqref{comp1} leads to the desired bound 
for the integral related to $J_1$. Proceeding in a similar way we may easily get the same estimate for the integral connected with $J_2$.
Hence to finish the proof, it is enough to verify \eqref{Iest} for $j=2$.

An application of Lemma \ref{lem:Omega} gives
\begin{align*}
\int_{|\eta|<t} I_2(\t,\t',\eta,t) \, d\eta
	& =
\frac{\big| V_t^{\ab}(\t) - V_t^{\ab}(\t') \big|}{V_t^{\ab}(\t')}.
\end{align*}
By the definition of $V_t^{\ab}(\t)$ and the fact that $|\t - \t'| \le t$, we get
\[
\big| V_t^{\ab}(\t) - V_t^{\ab}(\t') \big|
= 
\big| \mu_{\ab} (\t \wedge \t' - t , \t \vee \t' - t) 
- \mu_{\ab} ( \t \wedge \t' + t , \t \vee \t' + t) 
\big|,
\]
and further
\begin{align*}
\frac{\big| V_t^{\ab}(\t) - V_t^{\ab}(\t') \big|}{V_t^{\ab}(\t')}
	& \le 
\frac{\mu_{\ab}  (\t \wedge \t' - t , \t \vee \t' - t) }
{V_t^{\ab}(\t')}
+
\frac{\mu_{\ab} (\t \wedge \t' + t , \t \vee \t' + t) }
{V_t^{\ab}(\t')} \\
	& \equiv
K^{\ab}_1(\t,\t',t) + K^{\ab}_2(\t,\t',t).
\end{align*}
Since $\mu_{\ab}(B(x,r)) = \mu_{\ba}(B(\pi - x,r))$ and consequently
$K^{\ab}_2(\t,\t',t) = K^{\ba}_1(\pi - \t,\pi - \t',t)$, 
in order to finish the proof,
it suffices to verify that
\begin{equation*}
K^{\ab}_1(\t,\t',t)
\lesssim
\bigg( \frac{|\t -\t'|}{t} \bigg)^{2\gamma},
\qquad t \in (0,\pi], \quad \t,\t' \in (0,\pi), \quad |\t - \t'| \le t.
\end{equation*}
Taking into account \eqref{ball} we obtain
\begin{align*}
& \mu_{\ab}  (\t \wedge \t' - t , \t \vee \t' - t)  \\
	& \quad =
\chi_{ \{ \t \wedge \t' \ge t  \}}
\mu_{\ab} \bigg( B \bigg(\frac{\t + \t' - 2t}{2} , \frac{|\t -\t'|}{2} \bigg) \bigg) 
+
\chi_{ \{ \t \wedge \t' < t \le  \t \vee \t' \}}
\mu_{\ab} \bigg( B \bigg(\frac{\t \vee \t' - t}{2} , 
\frac{\t \vee \t' - t}{2} \bigg) \bigg)  \\
	& \quad \simeq
\chi_{ \{ \t \wedge \t' \ge t  \}}
|\t -\t'| (\t \vee \t' - t)^{2\a + 1} (\pi - \t \wedge \t' + t)^{2\b + 1} 
+
\chi_{ \{ \t \wedge \t' < t \le  \t \vee \t' \}}
(\t \vee \t' - t)^{2\a + 2}.
\end{align*}
Using once again \eqref{ball} and \eqref{comp1} we get
\begin{align*}
K^{\ab}_1(\t,\t',t)
& \simeq
\chi_{ \{ \t \wedge \t' \ge 2t  \}}
\frac{|\t -\t'|}{t}
+
\chi_{ \{ 2t > \t \wedge \t' \ge t  \}}
|\t -\t'| \frac{(\t \vee \t' - t)^{2\a + 1}}{t^{2\a + 2}} \\
	 & \qquad +
\chi_{ \{ \t \wedge \t' < t \le  \t \vee \t' \}}
\frac{ (\t \vee \t' - t)^{2\a + 2} }{t^{2\a + 2}}.
\end{align*}
The relevant bound for the first term is straightforward. To analyze the second one it is enough to use the estimates
\begin{align*}
\chi_{ \{ 2t > \t \wedge \t' \ge t  \}}
|\t - \t'|^{2\gamma + (1-2\gamma)}
(\t \vee \t' - t)^{2\a + 1}
	& \le
\chi_{ \{ 2t > \t \wedge \t'  \}}
|\t - \t' |^{2\gamma} 
(\t \vee \t' - t)^{2\a + 2 - 2\gamma} \\
	 & \lesssim
|\t - \t' |^{2\gamma} 
t^{2\a + 2 - 2\gamma}.
\end{align*}
The third term can be bounded by means of the inequality
\begin{align*}
\chi_{ \{ \t \wedge \t' < t \le  \t \vee \t' \}}
(\t \vee \t' - t)^{2\a + 2}
\le (\t \vee \t' - \t \wedge \t')^{2\a + 2}
= |\t - \t'|^{2\a + 2},
\end{align*}
and the fact that $2\a + 2 > 2\gamma$.
This finishes proving Lemma \ref{lem:Omegadiff}.

\end{proof}

\begin{lem}\label{lem:Omegaprime}
Assume that $\ab > -1$. 
Then there exists $\gamma=\gamma(\a,\b) \in (0,1/2]$ such that
\begin{align*}
\int_{|\eta|<t} \chi_{ \{ \t + \eta \in (0,\pi), \t' + \eta  \notin (0,\pi) \} } 
\ot  \, d\eta
\lesssim
\begin{cases}
        \Big( \frac{|\t -\t'|}{t} \Big)^{2\gamma}, & t \in (0,\pi] \\
        |\t -\t'|^{2\gamma}, &  t \ge \pi
    \end{cases}, 
		\qquad \t,\t' \in (0,\pi).
\end{align*}
Moreover, the estimate holds with any $\gamma \in (0,1/2]$ satisfying 
$\gamma \le \a \wedge \b + 1$.
\end{lem}

\begin{proof}
Let $\gamma \in (0,1/2]$ satisfying $\gamma \le \a \wedge \b + 1$ be fixed.
By the same arguments as those given in the beginning of the proof of Lemma \ref{lem:Omegadiff} we may assume that $0 < t \le \pi$ and $|\t - \t'| \le t$.
To proceed, by the definition of $\ot$ we get
\begin{align*}
& \int_{|\eta|<t} \chi_{ \{ \t + \eta \in (0,\pi), \t' + \eta  \notin (0,\pi) \} } 
\ot  \, d\eta \\
	& \quad =
\frac{1}{V_t^{\ab}(\t)} 
\int_{|\eta|<t} 
(\chi_{ \{ \t' + \eta \le 0 \} } + \chi_{ \{ \t' + \eta \ge \pi \}} )
\chi_{ \{ \t + \eta \in (0,\pi) \} } 
\Big( \ste \Big)^{2\a +1} \Big( \cte \Big)^{2\b +1}  \, d\eta \\
	& \quad \equiv
I^{\ab}_1(\t,\t',t) + I^{\ab}_2(\t,\t',t).
\end{align*}
Since $I^{\ab}_2(\t,\t',t) = I^{\ba}_1(\pi - \t,\pi - \t',t)$, in order to finish the proof, it is enough to check that 
\[
I^{\ab}_1 (\t,\t',t) \lesssim
 \bigg( \frac{|\t -\t'|}{t} \bigg)^{2\gamma}, \qquad t \in (0,\pi], 
\quad \t,\t' \in (0,\pi), \quad |\t - \t'| \le t. 
\]
Observe that
\begin{align*}
I^{\ab}_1(\t,\t',t) &= 
\chi_{ \{  \t > \t', \, t > \t'  \} }
\frac{\mu_{\ab}  (\t -t, \t -\t') }{V_t^{\ab}(\t)} \\
&=
\chi_{ \{  \t > t > \t'  \} }
\frac{\mu_{\ab} \big( B \big(\t -\frac{t+\t'}2, \frac{t -\t'}2 \big) \big) }{V_t^{\ab}(\t)} 
+
\chi_{ \{  t \ge \t > \t'  \} }
\frac{\mu_{\ab} \big( B \big( \frac{\t-\t'}2, \frac{\t-\t'}2 \big) \big) }{V_t^{\ab}(\t)}.
\end{align*}
Taking into account \eqref{ball} and then using \eqref{comp1} we obtain
\begin{align*}
I^{\ab}_1(\t,\t',t) & \simeq
\chi_{ \{ \t > t > \t'  \} }
\frac{ (t-\t') (\t-\t')^{2\a+1} }{t (\t+t)^{2\a+1}}
	+ 
\chi_{ \{ t \ge \t > \t'  \} }
\frac{  (\t-\t')^{2\a+2} }{t^{2\a+2}}
\lesssim
\bigg( \frac{|\t -\t'|}{t} \bigg)^{2\a +2},
\end{align*}
which finishes the proof because $2\a+2 \ge 2\gamma$.
\end{proof}

The following lemma comes into play when proving the smoothness conditions \eqref{sm1}, \eqref{sm2} for the kernels in question.

\begin{lem}\label{lem:Upstilde}
Assume that $\ab > -1$ and $W,s \in \R$ are fixed. Then for all $t \in (0,\pi]$ and $\t, \widetilde{\t},\vp \in (0,\pi)$ with $|\t - \vp| > 2|\t - \widetilde{\t}|$,
we have
\[
\Upsilon_{W,s}^{\ab}(t,\widetilde{\t},\vp) \simeq 
\Upsilon_{W,s}^{\ab}(t,\t,\vp).
\]
Analogously, for all $t \in (0,\pi]$ and $\t,\vp, \widetilde{\vp} \in (0,\pi)$ with $|\t - \vp| > 2|\vp - \widetilde{\vp}|$,
\[
\Upsilon_{W,s}^{\ab}(t,\t,\widetilde{\vp}) \simeq 
\Upsilon_{W,s}^{\ab}(t,\t,\vp).
\]
\end{lem}

\begin{proof}
For symmetry reasons it is enough to prove only the first part of this lemma. 
Further, taking into account the definition of $\Upsilon_{W,s}^{\ab}(t,\t,\vp)$ and 
the relation, see \cite[Lemma 4.6]{NoSj},
\[
q(\widetilde{\t},\vp,u,v) \simeq q(\t,\vp,u,v),
\qquad |\t-\vp|>2|\t-\widetilde{\t}|, \quad u,v \in [-1,1],
\]
we see that our task is reduced to showing that
\[
\stt + \svp \simeq \st + \svp, \qquad 
\ctt + \cvp \simeq \ct + \cvp, \qquad |\t-\vp|>2|\t-\widetilde{\t}|.
\]
Reflecting in $\pi/2$, it suffices to verify only the first relation. Moreover, with the aid of \eqref{compsin}, this is equivalent to checking that
\[
\widetilde{\t} + \vp \simeq \t +\vp, \qquad |\t-\vp|>2|\t-\widetilde{\t}|.
\]
This, however, is routine and follows by a simple application of the triangle inequality. 
\end{proof}

Finally, we are ready to prove Theorem \ref{thm:kerest}.

\begin{proof}[Proof of Theorem \ref{thm:kerest}.]
We first show the standard estimates for the kernel $S^{\ab}_{M,N}(\t,\vp)$. To begin with we prove the growth condition \eqref{gr}.
By Lemma \ref{lem:Ht1eta} (specified to $L = P = 0$) we have
\begin{align} \label{Sest}
|S^{\ab}_{M,N,\eta,t}(\t,\vp)|
	& \lesssim
\Big(
\chi_{ \{ t \in (0,\pi) \} } \Upsilon^{\ab}_{2M + 2N, 0}(t,\t,\vp) \\
	& \qquad + \nonumber
\chi_{ \{ t \ge \pi \} } 
\sup_{ \xi,\zeta \in (0,\pi) } \big| \partial_{t}^M \delta_{\xi}^N H_t^{\ab}(\xi,\zeta) \big|
\Big)
\chi_{ \{ \t + \eta  \in (0,\pi) \} }
\sqrt{\ot},
\end{align}
for $\t,\vp \in (0,\pi)$ and $(\eta,t) \in \Gamma$.
Consequently, using Lemma \ref{lem:Omega}, we infer that
\begin{align*}
\big\| S^{\ab}_{M,N,\eta,t}(\t,\vp) 
\big\|_{ L^2(\Gamma, t^{2M+2N-1}d\eta dt) }
	 & \lesssim
\big\| \Upsilon^{\ab}_{2M + 2N, 0}(t,\t,\vp) 
\big\|_{ L^2((0,\pi), t^{2M+2N-1} dt) } \\
	 & \qquad +
\big\| \sup_{ \xi,\zeta \in (0,\pi) } \big| \partial_{t}^M \delta_{\xi}^N H_t^{\ab}(\xi,\zeta) \big|
\big\|_{ L^2((\pi,\infty), t^{2M+2N-1} dt) }.
\end{align*}
Now an application of Lemma \ref{lem:finbridgep=2} (specified to $W=2M + 2N$, $s=0$)
and Lemma \ref{cor:tLp=2} (taken with $L = P = 0$ and $W=2M + 2N$) leads to the desired estimate.

We pass to proving the first smoothness condition. 
More precisely, we will show \eqref{sm1} with any fixed $\gamma \in (0,1/2]$ satisfying $\gamma < \a \wedge \b + 1$.
It is convenient to split the region of integration into four subsets, depending on whether 
$\t + \eta, \t' + \eta$ belong to $(0,\pi)$ or not. Let
\begin{align*}
\Gamma_1 & = \Gamma \cap \{ (\eta,t) : \t + \eta \in (0,\pi), 
\t' + \eta \in (0,\pi) \}, \\
\Gamma_2 & = \Gamma \cap \{ (\eta,t) : \t + \eta \in (0,\pi), 
\t' + \eta \notin (0,\pi) \}, \\
\Gamma_3 & = \Gamma \cap \{ (\eta,t) : \t + \eta \notin (0,\pi), 
\t' + \eta \in (0,\pi) \}, \\
\Gamma_4 & = \Gamma \cap \{ (\eta,t) : \t + \eta \notin (0,\pi), 
\t' + \eta \notin (0,\pi) \}. 
\end{align*}
Since the treatment of the integral over $\Gamma_4$ is trivial and the case of $\Gamma_3$ is analogous to $\Gamma_2$ 
(precisely, one should  use in addition Lemma \ref{lem:Upstilde} with $\widetilde{\t} = \t'$ in the parallel reasoning related to $\Gamma_3$), it suffices to focus only on the two essential cases.

\noindent
{\bf Case 1:} \textbf{The analysis related to} $\mathbf{L^2(\Gamma_1, t^{2M+2N-1}d\eta dt)}.$
By the triangle inequality we get
\begin{align*}
& \big| S^{\ab}_{M,N,\eta,t}(\t,\vp) - S^{\ab}_{M,N,\eta,t}(\t',\vp) \big| \\
	 & \quad \le
\Big| \partial_{t}^M \delta_{\psi}^N H_{t}^{\ab}(\psi,\vp) \big|_{\psi = \t + \eta}  -
\partial_{t}^M \delta_{\psi}^N H_{t}^{\ab}(\psi,\vp) \big|_{\psi = \t' + \eta} \Big|
\sqrt{\ott} \\
	 & \quad \qquad + 
\Big| \partial_{t}^M \delta_{\psi}^N H_{t}^{\ab}(\psi,\vp) \big|_{\psi = \t + \eta} \Big|
\Big| \sqrt{\ot} - \sqrt{\ott} \Big| \\
	 & \quad \equiv
I_1(\t,\t',\vp,\eta,t) + I_2(\t,\t',\vp,\eta,t).
\end{align*}
We will treat $I_1$ and $I_2$ separately. We begin with showing that
\begin{align*}
\| I_1(\t,\t',\vp,\eta,t) \|_{ L^2(\Gamma_1, t^{2M+2N-1}d\eta dt) }
\lesssim
\frac{|\t - \t'|}{|\t-\vp|} 
\; \frac{1}{\mu_{\ab}(B(\t,|\t-\vp|))}, \qquad |\t-\vp| > 2|\t - \t'|.
\end{align*}
By the Mean Value Theorem we get
\[
| I_1(\t,\t',\vp,\eta,t) |
\le
|\t - \t'| 
\Big| 
\partial_{\psi} \partial_{t}^M \delta_{\psi}^N H_{t}^{\ab}(\psi,\vp) \big|_{\psi = \widetilde{\t} + \eta}\Big|
\sqrt{\ott},
\]
where $\widetilde{\t}$ is a convex combination of $\t$ and $\t'$ that depends also on $\eta$ and $t$.
Now an application of Lemma \ref{lem:Ht1eta} (with $L = 0$, $P = 1$)
and then Lemma \ref{lem:Upstilde} gives
\begin{align*}
| I_1(\t,\t',\vp,\eta,t) |
& \lesssim
|\t - \t'|
\Big(
\chi_{ \{ t \in (0,\pi) \} } \Upsilon^{\ab}_{2M + 2N, 1}(t,\t,\vp)  \\
	 & \qquad + 
\chi_{ \{ t \ge \pi \} } 
\sup_{ \xi,\zeta \in (0,\pi) } \big| \partial_{\xi} \partial_{t}^M \delta_{\xi}^N H_t^{\ab}(\xi,\zeta) \big|
\Big)
\sqrt{\ott},
\end{align*}
provided that $|\t-\vp| > 2|\t - \t'|$ and $(\eta,t) \in \Gamma_1$.
Hence, with the aid of Lemma \ref{lem:Omega}, 
Lemma \ref{lem:finbridgep=2} (with $W = 2M + 2N$, $s=1$) and 
Lemma \ref{cor:tLp=2} (taken with $L = 0$, $P = 1$, $W = 2M + 2N$), the asserted norm estimate for $I_1$ follows.

Next we check that 
\[
\| I_2(\t,\t',\vp,\eta,t) \|_{ L^2(\Gamma_1, t^{2M+2N-1}d\eta dt) }
\lesssim
\bigg( \frac{|\t - \t'|}{|\t-\vp|} \bigg)^{\gamma}  
\; \frac{1}{\mu_{\ab}(B(\t,|\t-\vp|))}, \qquad |\t-\vp| > 2|\t - \t'|.
\]
Using Lemma \ref{lem:Ht1eta} (with $L = P = 0$) we arrive at the bound 
\begin{align*}
| I_2(\t,\t',\vp,\eta,t) |
& \lesssim
\Big(
\chi_{ \{ t \in (0,\pi) \} } \Upsilon^{\ab}_{2M + 2N, 0}(t,\t,\vp)  
	+ 
\chi_{ \{ t \ge \pi \} } 
\sup_{ \xi,\zeta \in (0,\pi) } \big| \partial_{t}^M \delta_{\xi}^N H_t^{\ab}(\xi,\zeta) \big|
\Big) \\
& \qquad \qquad \qquad \times
\Big| \sqrt{\ot} - \sqrt{\ott} \Big|,
\end{align*}
for $\t, \t', \vp \in (0,\pi)$ and $(\eta,t) \in \Gamma_1$.
Further, an application of Lemma \ref{lem:Omegadiff} leads to
\begin{align*}
\| I_2(\t,\t',\vp,\eta,t) \|_{ L^2(\Gamma_1, t^{2M+2N-1}d\eta dt) }
	 & \lesssim
|\t - \t'|^{\gamma}
\bigg(
\big\| \Upsilon^{\ab}_{2M + 2N, 0}(t,\t,\vp) 
\big\|_{ L^2((0,\pi), t^{2M+2N-2\gamma-1} dt) } \\
	 & \qquad +
\big\| \sup_{ \xi,\zeta \in (0,\pi) } \big| \partial_{t}^M \delta_{\xi}^N H_t^{\ab}(\xi,\zeta) \big|
\big\|_{ L^2((\pi,\infty), t^{2M+2N-1} dt) }
\bigg).
\end{align*}
Now,  
Lemma \ref{lem:finbridgep=2} (taken with $W = 2M + 2N - 2\gamma \ge 1$, $s=\gamma$; 
observe that $\Upsilon^{\ab}_{2M + 2N,0}(t,\t,\vp) = \Upsilon^{\ab}_{2M + 2N - 2\gamma, \gamma}(t,\t,\vp)$)
together with Lemma \ref{cor:tLp=2} (specified to $L = P = 0$, $W = 2M + 2N$) produces the required bound for $I_2$. The analysis associated with $\Gamma_1$ is finished.

\noindent
{\bf Case 2:} \textbf{The analysis related to} $\mathbf{L^2(\Gamma_2, t^{2M+2N-1}d\eta dt)}.$
Since $S^{\ab}_{M,N,\eta,t}(\t',\vp) = 0$ for $(\eta,t) \in \Gamma_2$, it suffices to verify that
\begin{equation}\label{Gamma2}
\big\| S^{\ab}_{M,N,\eta,t}(\t,\vp) 
\big\|_{ L^2(\Gamma_2, t^{2M+2N-1}d\eta dt) }
\lesssim
\bigg( \frac{|\t - \t'|}{|\t-\vp|} \bigg)^{\gamma} 
\; \frac{1}{\mu_{\ab}(B(\t,|\t-\vp|))},
\end{equation}
for $|\t-\vp| > 2|\t - \t'|$. 
Combining \eqref{Sest} with Lemma \ref{lem:Omegaprime} gives
\begin{align*}
\big\| S^{\ab}_{M,N,\eta,t}(\t,\vp) 
\big\|_{ L^2(\Gamma_2, t^{2M+2N-1}d\eta dt) }
	& \lesssim
|\t - \t'|^{\gamma}
\bigg(
\big\| \Upsilon^{\ab}_{2M + 2N, 0}(t,\t,\vp) 
\big\|_{ L^2((0,\pi), t^{2M+2N-2\gamma-1} dt) } \\
	 & \qquad +
\big\| \sup_{ \xi,\zeta \in (0,\pi) } \big| \partial_{t}^M \delta_{\xi}^N H_t^{\ab}(\xi,\zeta) \big|
\big\|_{ L^2((\pi,\infty), t^{2M+2N-1} dt) }
\bigg),
\end{align*}
which leads to \eqref{Gamma2}, see the norm estimate of $I_2$ above.

Next we verify the second smoothness condition \eqref{sm2} with $\gamma = 1$ and $S^{\ab}_{M,N}(\t,\vp)$. 
Using the Mean Value Theorem, Lemma \ref{lem:Ht1eta} (with $L = 1$, $P = 0$) and then Lemma \ref{lem:Upstilde}, we obtain
\begin{align*}
\big| S^{\ab}_{M,N,\eta,t}(\t,\vp) - S^{\ab}_{M,N,\eta,t}(\t,\vp') 
\big|
	 & \lesssim
|\vp - \vp'|
\Big(
\chi_{ \{ t \in (0,\pi) \} } \Upsilon^{\ab}_{2M + 2N, 1}(t,\t,\vp)  \\
	 & \qquad + 
\chi_{ \{ t \ge \pi \} } 
\sup_{ \xi,\zeta \in (0,\pi) } \big| \partial_{\zeta} \partial_{t}^M \delta_{\xi}^{N} H_t^{\ab}(\xi,\zeta) \big|
\Big)
\sqrt{\ot},
\end{align*}
provided that $|\t-\vp| > 2|\vp - \vp'|$, $\t + \eta \in (0,\pi)$ and $(\eta,t) \in \Gamma$. Now, using sequently Lemma \ref{lem:Omega}, Lemma \ref{lem:finbridgep=2} (applied with $W = 2M + 2N$, $s=1$)
and Lemma \ref{cor:tLp=2} (specified to $L = 1$, $P = 0$, $W = 2M + 2N$) we get the desired conclusion.
This finishes proving Theorem \ref{thm:kerest} for the kernel $S^{\ab}_{M,N}(\t,\vp)$.

The proof of the standard estimates for the kernel $\mathcal{S}^{\ab}_{M,N}(\t,\vp)$ is just a repetition of the reasoning given above. This is indeed the case because the estimates of certain derivatives of the Jacobi-Poisson kernel established in Lemmas \ref{lem:Ht1}, \ref{lem:Ht1eta} and \ref{cor:tLp=2} are the same for both kinds of higher order derivatives $\delta^N$ and $D^N$.
We leave the details to the reader.

\end{proof}

\end{document}